\documentclass[reqno,10pt]{amsart}
\usepackage{amsmath,amsfonts,amssymb,amsxtra,latexsym,amscd,enumerate,amsthm,verbatim}

\usepackage{hyperref}
\usepackage{graphicx}
\usepackage{cancel}

\usepackage{mathtools}
\mathtoolsset{showonlyrefs=true}

\usepackage[margin=1.3in]{geometry}
\numberwithin{equation}{section}

\def\t{\tau}

\def\R{\mathbb{R}}

\def\Z{\mathbb{Z}}

\def\Z{\mathbb{Z}}

\def\be{{\beta}}

\providecommand{\ip}[1]{\langle#1\rangle}
\providecommand{\abs}[1]{\left\lvert#1\right\rvert}
\providecommand{\norm}[1]{\left\|#1\right\|}

\def\be{\begin{equation}}
\def\ee{\end{equation}}
\newtheorem{theorem}{Theorem}[section]
\newtheorem{lemma}[theorem]{Lemma}

\newtheorem{proposition}[theorem]{Proposition}

\theoremstyle{remark}
\newtheorem{remark}[theorem]{Remark}

\def\beq{\begin{equation}}
\def\eeq{\end{equation}}

\def\beq{\begin{equation}}
\def\eeq{\end{equation}}

\begin{document}

\title{On the asymptotic behavior of solutions to the Vlasov-Poisson system}

\author{Alexandru D. Ionescu}
\address{Princeton University}
\email{aionescu@math.princeton.edu}
\author{Benoit Pausader}
\address{Brown University}
\email{benoit\_pausader@brown.edu}
\author{Xuecheng Wang}
\address{Tsinghua University}
\email{xuecheng@tsinghua.edu.cn}
\author{Klaus Widmayer}
\address{\'Ecole Polytechnique F\'ed\'erale de Lausanne}
\email{klaus.widmayer@epfl.ch}

\thanks{A.\ I.\ was supported in part by NSF grant DMS-1600028, B.\ P.\ was supported in part by NSF grant DMS-1700282, X.\ W.\ is supported by NSFC-11801299.}

\maketitle

\begin{abstract}

We prove small data modified scattering for the Vlasov-Poisson system in dimension $d=3$ using a method inspired from dispersive analysis. In particular, we identify a simple asymptotic dynamic related to the scattering mass.

\end{abstract}

\section{Introduction}

\subsection{The Vlasov-Poisson system}

We consider the Vlasov-Poisson system for a density function $f:\mathbb{R}^3_x\times\mathbb{R}^3_v\times\mathbb{R}_t\to\mathbb{R}_+$:
\begin{equation}\label{eq:VP_full}
\begin{cases}
\left(\partial_t+v\cdot\nabla_x\right)f-q\nabla_x\phi\cdot\nabla_vf=0, &q=\pm 1,\\
-\Delta_x\phi(x,t)=\int_{\mathbb{R}^3}f(x,v,t)dv,&\\
f(t=0,x,v)=f_0(x,v).
\end{cases}
\end{equation}
This model is relevant in plasma physics (usually for $q=-1$) and in astrophysics (for $q=1$); we refer to \cite{Glassey,Mouhot} for more background. In dimension $d=3$, solutions to \eqref{eq:VP_full} are global in time under rather mild assumptions \cite{LionsPerthame,Pfa}, but a complete understanding of their asymptotic behavior is still elusive. In the case of small data \cite{BardosDegond} provide decay estimates, and modified scattering was established in \cite{ChoiKwon}. Recently, these works have been revisited from different point of views \cite{HwangRendallVelazquez,Smu,Wang3,Wang4} with varying improvements.

The relationship between kinetic and dispersive equations, particularly the Schr\"odinger equation is classical, see e.g.\ \cite[Section 1.2]{LemMehRap} for a compelling presentation. A quantum analog of the Vlasov-Poisson system is the Hartree equation, which can be analyzed effectively using dispersive tools \cite{KatoPusateri}. In this paper we want to adapt classical methods from dispersive equations to recover a simple proof of small data/modified scattering for \eqref{eq:VP_full} based on energy estimates and convergence in a weaker norm. In particular, this allows to clarify the role of some of the assumptions and to isolate a particularly simple asymptotic dynamic. We hope that this framework will be useful when considering coupling of kinetic equations and other (particularly dispersive) equations (see \cite{Wang1,Wang2} for examples of such problems).

\subsection{Main result}

Assuming that the initial density is a nonnegative function $f\geq 0$ (as opposed to a measure), by the transport nature of $\eqref{eq:VP_full}$ this condition is propagated along the flow. We may thus introduce $\mu=\sqrt{f}$ and consider the more symmetric equation
\begin{equation}\label{VP}
\begin{split}
\left(\partial_t+v\cdot\nabla_x\right)\mu-q\nabla_x\phi\cdot\nabla_v\mu&=0,\qquad
-\Delta_x\phi=\int_{\mathbb{R}^3}\mu^2dv.
\end{split}
\end{equation}
Our main result can thus be stated as follows:
\begin{theorem}\label{MainTheorem}
There exists $\varepsilon^\ast>0$ such that for any $0<\varepsilon_0\le \varepsilon^\ast$, the following holds: if $\mu_0$ is a smooth initial data such that
\begin{equation}\label{InitialData}
\begin{split}
\mathcal{E}(\mu_0):=\Vert x\mu_0\Vert_{L^2_{x,v}}+\Vert \mu_0\Vert_{H^3_{x,v}}\le\varepsilon_0
\end{split}
\end{equation}
then there exists a unique solution to \eqref{VP} which is global and scatters. This solution satisfies
\begin{equation}\label{GlobalBounds}
\begin{split}
\Vert \gamma(t)\Vert_{L^\infty_vL^2_x}+\Vert\gamma(t)\Vert_{L^2_vH^3_x}\lesssim\varepsilon_0,\qquad \mathcal{E}(\gamma(t))\lesssim \varepsilon_0\ln^3\langle t\rangle(\ln\langle\ln\langle t\rangle\rangle)^{6},
\end{split}
\end{equation}
where $\gamma(x,v,t)=\mu(x+tv,v,t)$. In addition, letting
\begin{equation}
\begin{split}
m_\infty(v)&:=\lim_{t\to\infty}\Vert \mu(\cdot,-v,t)\Vert_{L^2_x}^2,\qquad
\widetilde{E}(v):=\int_{\mathbb{R}^3}\frac{\zeta}{\vert\zeta\vert^3}m_\infty(\zeta-v)d\zeta,
\end{split}
\end{equation}
we have modified scattering to a new density function
\begin{equation}\label{ModScatExplicit}
\begin{split}
\mu(x+tv+q\ln(t)\cdot\widetilde{E}(v),v,t)\to\gamma_\infty(x,v)\qquad\text{ in }L^\infty_vL^2_x\cap H^1_{x,v}.
\end{split}
\end{equation}
\end{theorem}

A few remarks are in order:
\begin{remark}

\noindent
\begin{enumerate}

\item This theorem is not new, except for the limited assumptions on the initial data and the explicit form of the asymptotic behavior \eqref{ModScatExplicit}. We feel however that the importance of this paper lies in the simplicity and versatility in the method developed. In addition it clarifies the relevance of various controls (velocity moment and vector fields seem less relevant, regularity in $v$ seems central).

\item Our assumptions are not comparable with the ones in \cite{BardosDegond,ChoiKwon}, but are weaker than in most recent works \cite{Smu,Wang3}. This is related to the fact that we rely on energy estimates rather than transport bounds which naturally give pointwise bounds without appealing to Sobolev inequality.

\item Under our assumptions without velocity moment, classical large data global existence \cite{LionsPerthame,Pfa} do not apply. In particular, our solutions can have infinite physical momentum and energy.

\item If one considers measure initial data (in particular {\it monokinetic} initial data), the asymptotic behavior can be radically different, thus some amount of regularity is needed.

\item Moving from the density $f$ to $\mu=\sqrt{f}$ has several advantages: $(i)$ it automatically accounts for the nonnegativity of the density, $(ii)$ it allows to separate further our functions from Dirac masses (the natural space is now $\mu\in L^2_{x,v}$ as opposed to $f\in L^1_{x,v}$), $(iii)$ it makes the analogy with cubic dispersive problems (in particular Hartree) more transparent. In order to take advantage of the Sobolev scale, one might want to work on $\widetilde{\mu}=f^{1/2p}$, $p\ge2$. However, smoothness of $\widetilde{\mu}$ then carries nontrivial implications for $f$.

\item In our setting, we consider solutions which are perturbations of the vacuum since by density we may assume that $\mu_0\in C^\infty_c(\mathbb{R}^3_{x}\times\mathbb{R}^3_v)$. There are other natural equilibria with nondecaying density (e.g.\ BGK solutions \cite{BGK}). The analysis of their perturbations is related to the study of the {\it Landau damping}, see \cite{BedMasMou,MouVil,NguGreRod} and relies on different ideas.

\item It is remarkable that one only gets logarithmic growth of the energy. In addition, for $1$ derivative, one obtains optimal growth in the sense that the upper and lower growth rates are equal up to a multiplicative constant (compare the third bound in \eqref{BootstrapCcl2} and \eqref{ModScatExplicit}).

\item The methods presented here have broad application for kinetic equations, however, in general, e.g.\ for relativistic models, we expect that control on velocity moments would be necessary. Indeed the change of variable associated to dispersion (see \eqref{DecRho}) would involve losses in $v$ (see \cite{Wang3}).

\end{enumerate}
\end{remark}

\subsection{Method}

Our method extends a series of works on the asymptotic description of small data solutions for dispersive and related equations \cite{GeMaSh,HPTV,IoPaWKG,IoPaEKG,KatoPusateri} which couple energy estimates with a refined scattering analysis in a weaker norm, here
\begin{equation}\label{DefZNorm}
\begin{split}
\Vert \mu\Vert_Z&:=\Vert \mu\Vert_{L^\infty_vL^2_x},
\end{split}
\end{equation}
which is associated to conservation laws of the resonant/asymptotic system. In particular, one can observe that this norm is invariant both for the free streaming and for the modified scattering flow. Our analysis then proceeds over a few steps.

We obtain refined dispersive estimates involving the (minimal) $Z$-norm which will have to be uniformly bounded along the evolution. This is done in Lemma \ref{DispLemma} below and can be thought as an analog of similar estimates for the Schr\"odinger evolution. As for the Hartree equation, the critical step is however to bound a quartic expression which for us appears in Subsection \ref{EEII}.

It turns out that for \eqref{VP}, it is relatively easy to obtain a uniform bound on the $Z$-norm and convergence of the {\it scattering mass}
\begin{equation}\label{ScatteringMass}
m_t(v):=\Vert\mu(t)\Vert_Z^2=\int_{\mathbb{R}^3}\mu^2(x,v,t)dx\to m_\infty(v).
\end{equation}

The scattering mass controls how much mass is ``seen'' by a frame advected by free streaming and controls the asymptotic dynamics. In particular it allows to define the effective electric field and characteristics
\begin{equation}
\begin{split}
\frac{dX}{dt}=V,\qquad\frac{dV}{dt}=-\frac{q}{t^2}\widetilde{E}(V),\qquad\widetilde{E}(\zeta):=\nabla_\zeta(-\Delta_\zeta^{-1})m_\infty,
\end{split}
\end{equation}
from which modified scattering follows.

Commuting \eqref{VP} with the corresponding operator, one sees in \eqref{CommEq} that the energy estimates separate into simple estimates which do not require very sharp bounds and are less related to decay on the one hand (Subsection \ref{EEI}) and energy estimates for the velocity regularity which require sharp control on the unknowns and provide the ``fuel'' for the decay (which is obtained by trading regularity in $v$) in Subsection \ref{EEII}.

This paper is organized as follows: in Section \ref{Notations} we introduce our notations and some estimates to control the electric field. The global existence part of Theorem \ref{MainTheorem} is proved through a bootstrap in Section \ref{NAI}. Finally in Section \ref{NAII}, we obtain the modified scattering.

\section{Notations and preliminary estimates}\label{Notations}

In the following, since all our functions are evaluated at $t$, we have suppressed the explicit dependence in $t$ fo the functions $\gamma,\phi$. In order to be more thrifty in $v$ derivatives, we will introduce Littlewood-Paley projectors for $v$-regularity: for $\varphi$ a typical Littlewood-Paley bump function and $C\in 2^{\mathbb{Z}}$ a dyadic integer, we define
\begin{equation}\label{LPv}
\begin{split}
 P_C^v=\mathcal{F}^{-1}_{\theta\to v} \varphi(C^{-1}\theta)\mathcal{F}_{v\to\theta}
\end{split}
\end{equation}
Using Bernstein's inequality $\Vert h_M\Vert_{L^2_xL^\infty_v}\lesssim M^{\frac{3}{2}}\Vert h_M\Vert_{L^2_{x,v}}$, we can observe that the $Z$-norm defined in \eqref{DefZNorm} is bounded by the energy:
\begin{equation}\label{InterpolationZNorm}
\begin{split}
\norm{h}_Z&\lesssim \norm{h}_{L^2_xL^\infty_v}\lesssim \Vert h \Vert_{L^2_xH^2_v},\\
\Vert \nabla_v h_C\Vert_Z&\lesssim \min\{C\Vert h\Vert_Z,C^{-\frac{1}{2}}\Vert h\Vert_{L^2_xH^3_v}\},\qquad\Vert \nabla_v h \Vert_Z\lesssim \Vert h\Vert_{L^2_xH^3_v}^\frac{2}{3}\Vert h\Vert_Z^\frac{1}{3},
\end{split}
\end{equation}
where we have written $h_C=P_C^vh$ with notations from \eqref{LPv}. Besides, using that
\begin{equation*}\label{GainDifferenceLoc}
\begin{split}
h_M(a,\frac{x-a}{t})-h_M(a,\frac{x}{t})=\int_{\theta=0}^1\frac{a^j}{t}\partial_{v^j}h_M(a,\frac{x}{t}-\theta\frac{a}{t})d\theta
 \end{split}
 \end{equation*}
  we obtain after summing over $M$,
\begin{equation}\label{GainDifference}
\begin{split}
\Vert h(a,\frac{x-a}{t})-h(a,\frac{x}{t})\Vert_{L^2_{a}L^\infty_x}&\lesssim\sum_M\min\{t^{-1}M^\frac{5}{2}\Vert x \gamma\Vert_{L^2_{x,v}},M^{-\frac{1}{2}}\Vert \gamma\Vert_{L^2_xH^2_v}\}\\
&\lesssim t^{-\frac{1}{6}}\left\{\Vert h\Vert_{L^2_xH^2_v}+\Vert xh\Vert_{L^2_{x,v}}\right\}.
\end{split}
\end{equation}
Because of bounds like \eqref{GainDifference}, it turns out that our estimates are more simply stated using a variation of the $Z$-norm:
\begin{equation}\label{ZPrimeNorm}
\Vert f\Vert_{Z^\prime}:=\Vert f\Vert_Z+\langle t\rangle^{-\frac{1}{100}}\left\{\Vert f\Vert_{H^2_{x,v}}+\Vert xf\Vert_{L^2_{x,v}}\right\}.
\end{equation}

For the proof of our multilinear estimates we will use the representation
\begin{equation}\label{DecR}
\begin{split}
\frac{1}{\vert x\vert^p}&=c_p(\chi)\int_{R=0}^\infty R^{-p}\chi(R^{-1}\vert x\vert)\frac{dR}{R},
\end{split}
\end{equation}
valid for all $\chi\in \mathcal{S}$ and $p>0$.

\subsection{Dispersive estimates}
We now study the decay properties of a particle density distribution $h$ under the linear flow of \eqref{VP}.

\begin{lemma}\label{DispLemma}
For any $x\in \mathbb{R}^3$, there holds that
\begin{equation}
\begin{split}
0\le \rho_{[h]}(x,t):=\int_{\mathbb{R}^3} h^2(x-tv,v)dv&\lesssim \ip{t}^{-3}\norm{h}_{Z^\prime}^2,
\end{split}
\end{equation}
and 
\begin{equation}\label{BoundElectricFieldSharp}
\begin{split}
\Vert \nabla_x\phi_{[h]}(x,t)\Vert_{L^\infty_x}&\lesssim \ip{t}^{-2}[\norm{h}_{Z^\prime}^2+\norm{h}_{L^2_{x,v}}^2],\qquad -\Delta\phi_{[h]}=\rho_{[h]}.
\end{split}
\end{equation}
Moreover,
\begin{align}
 \norm{\partial_x^\alpha\nabla_x\phi_{[h]}(x,t)}_{L^\infty_x}&\lesssim \ip{t}^{-2-\abs{\alpha}}\norm{h}_{L^2_xH^{\abs{\alpha}+1}_v}^2,
 \quad &\abs{\alpha}\leq 2,\label{eq:DispLemmaLinf}\\
 \norm{\partial_x^\alpha\nabla_x\phi_{[h]}(x,t)}_{L^2_x}&\lesssim \ip{t}^{-\frac{1}{2}-\abs{\alpha}}\norm{h}_{L^2_xH^{\vert\alpha\vert}_v}^2, \quad &\abs{\alpha}\ge 2. \label{eq:DispLemmaL2}
\end{align} 

\end{lemma}

\begin{proof}[Proof of Lemma \ref{DispLemma}]
Assume without loss of generality that $t\geq 1$. We change variables and rewrite
\begin{equation}\label{DecRho}
\begin{split}
\rho_{[h]}(x,t)&=\int_{\mathbb{R}^3} h(x-tv,v)^2dv=t^{-3}\int_{\mathbb{R}^3}h^2(a,\frac{x-a}{t})da=t^{-3}\int_{\mathbb{R}^3}h^2(a,\frac{x}{t})da+J(x,t),\\
J(x,t)&:=t^{-3}\int_{\mathbb{R}^3}\left\{h^2(a,\frac{x-a}{t})-h^2(a,\frac{x}{t})\right\}da.
\end{split}
\end{equation}
The principal part can be directly bounded in terms of the $Z$ norm,
\begin{equation}
 \int_{\mathbb{R}^3}h^2(a,\frac{x}{t})da\leq\norm{h}_Z^2,
\end{equation}
so it suffices to show the decay of $J$. To this end we observe that
\begin{equation}\label{JLowerOrder}
\begin{split}
\vert J(x,t)\vert &\le t^{-3}\Vert h\Vert_{L^2_xL^\infty_v}\Vert h(a,\frac{x-a}{t})-h(a,\frac{x}{t})\Vert_{L^2_{a}}\lesssim t^{-3-\frac{1}{10}}\norm{h}_{Z'}^2,
\end{split}
\end{equation}
where in the last inequality we used \eqref{GainDifference}.

\medskip

The second inequality uses similar ideas. Using \eqref{DecR}, we can decompose
\begin{equation}\label{IntegralElectricField}
\begin{split}
\nabla\phi_{[h]}(x,t)&=\frac{1}{4\pi}\int_{R=0}^\infty\frac{dR}{R^2}\iint_{\mathbb{R}^3_{y,u}}\nabla_y\left\{\varphi(R^{-1}\vert x-y\vert)\right\}h^2(y-tu,u)dydu=:\int_{R=0}^\infty\Phi_R(x,t)\frac{dR}{R^2}.
\end{split}
\end{equation}
On the one hand, we see that
\begin{equation}\label{eq:disp_hifreq}
\begin{split}
\vert \partial_x^\alpha\Phi_R(x,t)\vert&\lesssim R^{-1-\vert\alpha\vert}\iint_{\mathbb{R}^3_{y,u}}h^2(y-tu,u)dydu\lesssim R^{-1}\Vert h\Vert_{L^2_{x,v}}^2,
\end{split}
\end{equation}
which is good enough for large $R\ge t$. On the other hand, for small $R$ we have
\begin{equation}\label{eq:disp_lowfreq}
\begin{split}
 \Phi_R(x,t) &= \iint_{\mathbb{R}^3_{z,u}}\nabla_z\left\{\varphi(R^{-1}\vert z\vert)\right\}h^2(x-z-tu,u)dzdu\\
 &= t^{-3}\iint_{\mathbb{R}^3_{z,a}}\nabla_z\left\{\varphi(R^{-1}\vert z\vert)\right\}h^2(a,\frac{x-z-a}{t})dzda.
\end{split}
\end{equation}
Taking derivatives, we see that
\begin{equation}\label{AddedBoundMonday1}
\begin{split}
 \partial_x^\alpha\Phi_R(x,t)&= t^{-3-\vert\alpha\vert}\sum_{\abs{\alpha_1}+\abs{\alpha_2}=\vert\alpha\vert,\,\vert\alpha_1\vert\le\vert\alpha_2\vert}c_{\alpha_1,\alpha_2}\Phi_R^{\alpha_1,\alpha_2}(x,t),\\
\Phi_R^{\alpha_1,\alpha_2}(x,t)&:=\iint_{\mathbb{R}^3_{z,a}}\nabla_z\left\{\varphi(R^{-1}\vert z\vert)\right\}\partial_v^{\alpha_1}h(a,\frac{x-z-a}{t})\partial_v^{\alpha_2}h(a,\frac{x-z-a}{t})dzda.
 \end{split}
\end{equation}
If $\vert\alpha\vert=0$, we estimate, with \eqref{GainDifference},
\begin{equation*}
\begin{split}
\vert \Phi_R^{0,0}\vert&\lesssim \Vert \nabla_z\varphi(R^{-1}\vert z\vert)\Vert_{L^1_z}  \Vert h(a,\frac{x-z-a}{t})\Vert_{L^\infty_{z}L^2_a}\Vert h(a,\frac{x-z-a}{t})\Vert_{L^\infty_zL^2_a}\lesssim R^2\Vert h\Vert_{Z^\prime}^2,
\end{split}
\end{equation*}
while if $\vert\alpha\vert\ge 1$, using H\"older's inequality, we find that
\begin{equation*}
\begin{split}
\vert \Phi_R^{\alpha_1,\alpha_2}\vert&\lesssim \Vert \nabla_z\varphi(R^{-1}\vert z\vert)\Vert_{L^\frac{6}{5}_z}  \Vert \partial^{\alpha_1}_vh(a,\frac{x-z-a}{t})\Vert_{L^\infty_{z}L^2_a}\Vert\partial_v^{\alpha_2}h(a,\frac{x-z-a}{t})\Vert_{L^6_zL^2_a}\\
&\lesssim R^\frac{3}{2}t^\frac{1}{2}\Vert h\Vert_{L^2_xH^{\vert\alpha_1\vert+2}_v}\Vert h\Vert_{L^2_xH^{\vert\alpha_2\vert+1}_v}.
\end{split}
\end{equation*}
Integrating the above bounds for $R\le t$ and \eqref{eq:disp_hifreq} for $R\ge t$ in \eqref{IntegralElectricField}, we obtain \eqref{BoundElectricFieldSharp} and \eqref{eq:DispLemmaLinf}. Finally, for $\abs{\alpha}\ge 2$, we estimate in $L^2$. For $R\geq t$ it suffices to note that
\begin{equation}
\begin{aligned}
 \norm{\partial_x^\alpha\Phi_R(x,t)}_{L^2_x}&\lesssim \norm{\iint_{\mathbb{R}^3_{y,u}}\partial_x^\alpha\nabla_y\left\{\varphi(R^{-1}\vert x-y\vert)\right\}h^2(y-tu,u)dydu}_{L^2_x}\lesssim R^{-\vert\alpha\vert +\frac{1}{2}}\norm{h}_{L^2_{x,v}}^2.
\end{aligned} 
\end{equation}
For $t\leq R$ we use H\"older and Sobolev inequalities in \eqref{AddedBoundMonday1} to find (for $2/p_j=\vert\alpha_j\vert/\vert\alpha\vert$) that
\begin{equation*}
\begin{split}
 \abs{\Phi_R^{\alpha_1,\alpha_2}(x,t)}&\lesssim t^{-\vert\alpha\vert-3}\Vert \nabla_z\left\{\varphi(R^{-1}\vert z\vert)\right\}\Vert_{L^2_z}\Vert \partial_v^{\alpha_1}h(a,\frac{z}{t})\Vert_{L^2_aL^{p_1}_z}\Vert\partial_v^{\alpha_2}h(a,\frac{z}{t})\Vert_{L^2_aL^{p_2}_z} \lesssim t^{-\vert\alpha\vert-\frac{3}{2}}R^{\frac{1}{2}}\norm{h}_{L^2_xH^{\vert\alpha\vert}_v}^2,
 \end{split}
\end{equation*}
while estimating in $L^1_x$ yields
\begin{equation}
 \norm{\partial_x^\alpha\Phi_R(x,t)}_{L^1_x}\lesssim t^{-\vert\alpha\vert}\norm{\nabla_z\left\{\varphi(R^{-1}\vert z\vert)\right\}}_{L^1_x}\norm{h}_{L^2_xH^{\vert\alpha\vert}_v}^2\lesssim t^{-\vert\alpha\vert}R^2\norm{h}_{L^2_xH^{\vert\alpha\vert}_v}^2.
\end{equation}
By interpolation it thus follows that
\begin{equation}
 \norm{\partial_x^\alpha\Phi_R(x,t)}_{L^2_x}\lesssim t^{-\vert\alpha\vert-\frac{3}{4}}R^\frac{5}{4}\norm{h}_{L^2_xH^{\vert\alpha\vert}_v}^2,
\end{equation}
and hence
\begin{equation}
 \norm{\partial_x^\alpha\nabla_x\phi(t,x)}_{L^2}\lesssim \int_0^\infty \min\{t^{-\vert\alpha\vert-\frac{3}{4}}R^\frac{5}{4},R^{-\vert\alpha\vert+\frac{1}{2}}\} \frac{dR}{R^2}\cdot\norm{h}_{L^2_xH^{\vert\alpha\vert}_v}^2\lesssim t^{-\vert\alpha\vert-\frac{1}{2}}\norm{h}_{L^2_xH^{\vert\alpha\vert}_v}^2.
\end{equation}

\end{proof}

\begin{remark}
It is important to have a sharp control on the Electric field as in \eqref{BoundElectricFieldSharp}. The first bound on $\rho$ is here essentially for motivation to help clarify the relationship with the Schr\"odinger equation. The decomposition \eqref{DecRho} with \eqref{JLowerOrder} is one of the main motivations for the definition of the $Z$-norm. It can be compared to Fraunhoffer's inequality for the Schr\"odinger flow:
\begin{equation*}
\begin{split}
\left(e^{-it\Delta}f\right)(x)=\frac{e^{-i\frac{\vert x\vert^2}{4t}}}{(2\pi it)^\frac{d}{2}}\widehat{f}(-\frac{x}{2t})+O_{L^\infty}(t^{-\frac{d}{2}-\frac{1}{4}})
\end{split}
\end{equation*}
valid whenever $f\in\mathcal{S}$, which leads to the natural definition $\Vert f\Vert_Z=\Vert\widehat{f}\Vert_{L^\infty}$ for NLS, see \cite{GeMaSh,KatoPusateri}.
\end{remark}

\section{Nonlinear analysis I: Bootstrap of the norms}\label{NAI}
We first integrate the linear flow and define
\begin{equation}
\begin{split}
\gamma(x,v,t)&:=\mu(x+tv,v,t),\qquad \rho(x,t):=\int_{\mathbb{R}^3}\mu^2(x,v,t)dv=\int_{\mathbb{R}^3} \gamma^2(x-tv,v,t)dv,
\end{split}
\end{equation}
and we obtain the new equation
\begin{equation}\label{VPNew}
\begin{aligned}
\partial_t\gamma(x,v)&=q\nabla_x\phi(x+tv)\cdot\left\{\nabla_v-t\nabla_x\right\}\gamma(x,v),\qquad
-\Delta_x\phi=\rho.
\end{aligned}
\end{equation}

We can now state our main bootstrap proposition from which global existence and boundedness easily follow.

\begin{proposition}\label{prop:sol}
Let $\delta>0$ be a fixed small constant. There exists $\varepsilon^\ast>0$ such that for all $0<\varepsilon_0\le \varepsilon\le\varepsilon^\ast$, the following holds. Assume that $\gamma$ solves \eqref{VPNew} on $0\le t\le T$, with initial data $\mu_0$ satisfying \eqref{InitialData} and obeys the bound
\begin{equation}\label{BootstrapAss}
\begin{split}
\Vert \gamma\Vert_{L^2_{x,v}}&\le \varepsilon,\\
\Vert \gamma\Vert_{Z}+\Vert \gamma\Vert_{L^2_vH^3_x}&\le\varepsilon,\\
\Vert  \gamma\Vert_{L^2_xH^3_{v}}+\Vert  x\gamma\Vert_{L^2_{x,v}}&\le\varepsilon\langle t\rangle^{\delta},
\end{split}
\end{equation}
then there holds that
\begin{equation}\label{BootstrapAss1}
\begin{alignedat}{2}
\Vert\nabla_x\phi\Vert_{L^\infty_x}&\le C_1 \ip{t}^{-2}\varepsilon^2,\\
\Vert \partial_x^\alpha \nabla_x\phi\Vert_{L^\infty_x}&\le C_1\ip{t}^{-2-\abs{\alpha}+2\delta}\varepsilon^2,\qquad &1\leq\abs{\alpha}\leq 2,\\
\Vert \partial_x^\alpha \nabla_x\phi\Vert_{L^2_x}&\le C_1 \ip{t}^{-\frac{1}{2}-\abs{\alpha}+2\delta}\varepsilon^2,\qquad &2\leq\abs{\alpha}\leq 3,\\
\end{alignedat}
\end{equation}
and
\begin{equation}\label{BootstrapCcl2}
\begin{split}
\Vert \gamma\Vert_{L^2_{x,v}}&\le \varepsilon_0,\\
\Vert \gamma\Vert_{Z}+\norm{\nabla_x\gamma}_Z+\Vert \gamma\Vert_{L^2_vH^3_x}&\le\varepsilon_0+C_2\varepsilon^3,\\
\norm{\gamma}_{L^2_xH^1_v}+\Vert  x\gamma\Vert_{L^2_{x,v}}&\le\varepsilon_0+C_2\varepsilon^3\ln \langle t\rangle,\\
\norm{\gamma}_{L^2_xH^{\vert\alpha\vert}_v}&\leq\varepsilon_0+C_2\varepsilon^\frac{5}{2}(\ln \langle t\rangle)^{\vert\alpha\vert}\cdot(\ln\langle\ln\langle t\rangle\rangle)^{2\vert\alpha\vert},\qquad 2\le \vert\alpha\vert\le 3\\
\end{split}
\end{equation}
for some universal constant $C_1,C_2=C_2(\delta)$. In addition, the scattering mass defined in \eqref{ScatteringMass} converges uniformly to a limit $m_\infty(v)\in L^1_v\cap L^\infty_v$.
\begin{equation}\label{ConvergenceScatteringMass}
 \Vert m_t-m_\infty\Vert_{L^\infty_v}\lesssim \varepsilon^3\langle t\rangle^{-\frac{1}{2}}.
\end{equation}

\end{proposition}

Lemma \ref{DispLemma} gives \eqref{BootstrapAss1}. The first norm in \eqref{BootstrapCcl2} is trivially controlled by conservation laws (see \eqref{DivergenceFree} below). We give the bounds for the $Z$ norms in Subsection \ref{ZNorm}. Subsections \ref{EEI} and \ref{EEII} then demonstrates the energy estimates, including the most delicate terms: derivatives in velocity.

\begin{remark}[Regarding the growth rates of derivatives in $v$]

We have made some effort to obtain almost sharp bounds in the $v$ derivative; a slightly simpler analysis would have allowed to propagate slow polynomial growth. For one $v$ derivative, \eqref{BootstrapCcl2} gives the sharp growth rate $\ln(t)$, whereas for two or three derivatives there is an additional factor of $(\ln(\ln(t)))^p$. An inspection of the proof shows that if one were to propagate higher regularity in $v$, this loss could be relegated to higher orders of derivatives in $v$. We hope that the fact that we can obtain relatively simply optimal upper bounds for the first derivative illustrates the strength of our method.
\end{remark}

\subsection{Propagation of the $Z$-norm}\label{ZNorm}\label{NAZ}

We can directly compute that, for fixed $v\in\mathbb{R}^3$,
\begin{equation}
\begin{split}
\frac{1}{2}\frac{d}{dt}\int \gamma^2 dx&=q\int_{\mathbb{R}^3}\nabla_x\phi(x+tv)\gamma(x,v)\nabla_v\gamma(x,v) dx+q\frac{t}{2}\int_{\mathbb{R}^3}\Delta_x\phi(x+tv)\gamma^2(x,v)dx
\end{split}
\end{equation}
from which we deduce that
\begin{equation}
\begin{split}
\vert\Vert \gamma(t)\Vert_{Z}^2-\Vert\gamma_0\Vert_Z^2\vert\lesssim \int_{s=0}^t\Vert\nabla_x\phi(s)\Vert_{L^\infty_x}\Vert \gamma(s)\Vert_Z\Vert\nabla_v\gamma(s)\Vert_Z ds + \int_{s=0}^ts\Vert \rho(s)\Vert_{L^\infty}\Vert\gamma(s)\Vert_Z^2ds.
\end{split}
\end{equation}
Using \eqref{BootstrapAss1}, this leads to the bound of the $Z$ norm in \eqref{BootstrapCcl2} and to \eqref{ConvergenceScatteringMass}.

\subsubsection{Higher order $Z$-norm}

We can even control higher regularity in $Z$. Using \eqref{CommEq}, we see that
\begin{equation*}
\begin{split}
\frac{1}{2}\frac{d}{dt}\int_{\mathbb{R}^3_x}(\partial_{x^j}\gamma)^2dx&=q\int_{\mathbb{R}^3_x}\nabla_x\phi(x+tv)\partial_{x^j}\gamma\nabla_v\partial_{x^j}\gamma dx+\frac{t}{2}\int_{\mathbb{R}^3_x}\Delta_x\phi(x+tv)(\partial_{x^j}\gamma)^2dx\\
&\quad+q\int_{\mathbb{R}^3_x}\nabla_x\partial_{x^j}\phi(x+tv)\cdot\nabla_v\gamma\cdot\partial_{x^j}\gamma\,dx-qt\int_{\mathbb{R}^3_x}\nabla_x\partial_{x^j}\phi(x+tv)\cdot\nabla_x\gamma\cdot\partial_{x^j}\gamma dx.
\end{split}
\end{equation*}
All but the first terms can be controlled as before. The first term requires a little more work since it contains $2$ derivatives; however, we can still integrate the $x$-derivative by parts to move it to a more favorable position. To decide when to do it, we use the Littlewood-Paley decomposition \eqref{LPv} to decompose
\begin{equation}
 \int_{\R^3}\nabla_x\phi(x+tv)\partial_{x^j}\gamma\nabla_v\partial_{x^j}\gamma dx=\Big(\sum_{C_1> C_2}+\sum_{C_1\le C_2}\Big) \int_{\R^3}\nabla_x\phi(x+tv)\partial_{x^j}\gamma_{C_1}\nabla_v\partial_{x^j}\gamma_{C_2} dx,
\end{equation}
where on the one hand we estimate
\begin{equation}\label{EstimateSum1}
\begin{aligned}
 &\abs{\sum_{C_1>C_2}\int_{\R^3}\nabla_x\phi(x+tv)\partial_{x^j}\gamma_{C_1}\nabla_v\partial_{x^j}\gamma_{C_2} dx}\lesssim \sum_{C_1>C_2} \norm{\nabla_x\phi}_{L^\infty} \norm{\partial_{x^j}\gamma_{C_1}}_{L^\infty_vL^2_x}\norm{\nabla_v\partial_{x^j}\gamma_{C_2}}_{L^\infty_vL^2_x}\\
 &\qquad \lesssim \norm{\nabla_x\phi}_{L^\infty} \sum_{C_1>C_2} C_1^{\frac{3}{2}}\norm{\partial_{x^j}\gamma_{C_1}}_{L^2_{x,v}}C_2^{1+\frac{3}{2}}\norm{\partial_{x^j}\gamma_{C_2}}_{L^2_{x,v}}\\
 &\qquad\lesssim\norm{\nabla_x\phi}_{L^\infty} \sum_{C_1>C_2} \left(\frac{C_2}{C_1}\right)^{\frac{1}{2}}C_1^2\norm{\partial_{x^j}\gamma_{C_1}}_{L^2_{x,v}}C_2^2\norm{\partial_{x^j}\gamma_{C_2}}_{L^2_{x,v}}\lesssim\norm{\nabla_x\phi}_{L^\infty}\norm{\gamma}_{H^1_xH^2_v}^2.
\end{aligned} 
\end{equation}
To control the second sum, we integrate by parts in $x$ to get, when $C_1\le C_2$:
\begin{equation*}
\begin{split}
&\int_{\mathbb{R}^3}\nabla_x\phi(x+tv)\partial_{x^j}\gamma_{C_1}\nabla_v\partial_{x^j}\gamma_{C_2} dx\\
=&-\int_{\mathbb{R}^3}\partial_{x^j}\nabla_x\phi(x+tv)\partial_{x^j}\gamma_{C_1}\nabla_v\gamma_{C_2} dx-\int_{\mathbb{R}^3}\nabla_x\phi(x+tv)\partial_{x^j}^2\gamma_{C_1}\cdot \nabla_v\gamma_{C_2} dx.
\end{split}
\end{equation*}
The first term leads to a simple sum as before:
\begin{equation}
\begin{aligned}
 \abs{\sum_{C_1\leq C_2}\int_{\R^3}\partial_{x^j}\nabla_x\phi(x+tv)\partial_{x^j}\gamma_{C_1}\nabla_v\gamma_{C_2}dx}\lesssim \norm{\partial_{x^j}\nabla_x\phi}_{L^\infty}\norm{\gamma}_{H^1_xH^2_v}\norm{\gamma}_{L^2_xH^3_v},\\
\end{aligned}
\end{equation}
and we can sum the last term as in \eqref{EstimateSum1} to get
\begin{equation}
\begin{aligned}
 &\abs{\sum_{C_1\leq C_2}\int_{\R^3}\nabla_x\phi(x+tv)\partial_{x^j}^2\gamma_{C_1}\nabla_v\gamma_{C_2}dx}\lesssim\sum_{C_1\leq C_2}\norm{\nabla_x\phi}_{L^\infty}\norm{\partial_{x^j}^2\gamma_{C_1}}_{L^\infty_vL^2_x}\norm{\nabla_v\gamma_{C_2}}_{L^\infty_vL^2_x}\\
 &\lesssim\norm{\nabla_x\phi}_{L^\infty}\sum_{C_1\leq C_2}\left(\frac{C_1}{C_2}\right)^{\frac{1}{2}}C_1\norm{\partial_{x^j}^2\gamma_{C_1}}_{L^2_{x,v}}C_2^3\norm{\gamma_{C_2}}_{L^2_{x,v}}\lesssim \norm{\nabla_x\phi}_{L^\infty}\norm{\gamma}_{H^2_xH^1_v}\norm{\gamma}_{L^2_xH^3_v}.
\end{aligned}
\end{equation}
Combining the estimates with \eqref{BootstrapAss1}, we obtain a control of $\nabla_x\gamma$ in $Z$ as in \eqref{BootstrapCcl2}.

\subsection{Energy Estimates I: simpler energies}\label{EEI}

We can do energy estimates based on $L^2_{x,v}$ norm. Considering directly \eqref{VPNew} and deriving with respect to $x$, we find that (for $\gamma_{x^j}=\partial_{x^j}\gamma$ and $\gamma_{v^j}=\partial_{v^j}\gamma$)
\begin{equation}\label{CommEq}
\begin{split}
\partial_t\left\{x_j\gamma\right\}-q\nabla_x\phi(x+tv)\cdot\left\{\nabla_v-t\nabla_x\right\}\left\{x_j\gamma\right\}&=qt \partial_j\phi(x+tv)\cdot\gamma,\\
\partial_t\gamma_{x^j}-q\nabla_x\phi(x+tv)\cdot\left\{\nabla_v-t\nabla_x\right\}\gamma_{x^j}&=q\nabla_x\partial_{x^j}\phi(x+tv)\cdot\left\{\nabla_v-t\nabla_x\right\}\gamma,\\
\partial_t\gamma_{v^j}-q\nabla_x\phi(x+tv)\cdot\left\{\nabla_v-t\nabla_x\right\}\gamma_{v^j}&=qt\nabla_x\partial_{x^j}\phi(x+tv)\cdot\left\{\nabla_v-t\nabla_x\right\}\gamma.
\end{split}
\end{equation}
We can use this to control the rest of the norms in \eqref{BootstrapCcl2}. Note that since
\begin{equation}\label{DivergenceFree}
\hbox{div}_{x,v}(V)=0,\qquad V(x,v)=(-t\nabla_x\phi(x+tv,t),\nabla_x\phi(x+tv))\in\mathbb{R}^3_x\times\mathbb{R}^3_v
\end{equation}
 the left-hand side is conservative and hence only the righthand side contributes to changes in $L^2_{x,v}$ norms.

\subsubsection{Position}

We deduce from \eqref{CommEq} that
\begin{equation}
\begin{split}
\frac{d}{dt}\Vert x\gamma\Vert_{L^2_{x,v}}^2&\lesssim t\Vert\nabla_x\phi\Vert_{L^\infty_x}\Vert \gamma\Vert_{L^2}\Vert x\gamma\Vert_{L^2}.
\end{split}
\end{equation}
Using \eqref{BootstrapAss}-\eqref{BootstrapAss1} and Gr\"onwall, we obtain that
\begin{equation}
\begin{split}
\Vert x\gamma(t)\Vert_{L^2_{x,v}}&\le  \Vert x\gamma_0\Vert_{L^2_{x,v}}+C\varepsilon^2\Vert\gamma\Vert_{L^2_{x,v}} \ln t,
\end{split}
\end{equation}
which leads to the control of $x\gamma$ in \eqref{BootstrapCcl2}.

\subsubsection{Spatial regularity}

Commuting again \eqref{CommEq}, we see that for $1\le \abs{\alpha}\leq 2$
\begin{equation}
\begin{split}
\frac{1}{2}\frac{d}{dt}\Vert \partial_x^\alpha\gamma\Vert_{L^2}^2
&=q\sum_{\substack{\beta_1+\beta_2=\alpha,\,\,\vert\beta_2\vert<\vert\alpha\vert}}\iint \nabla_x\partial^{\beta_1}_x\phi(x+tv)\cdot\left\{\nabla_v-t\nabla_x\right\}\partial^{\beta_2}_x\gamma\cdot\partial_x^\alpha\gamma\\
&\lesssim \sum_{\substack{\beta_1+\beta_2=\alpha,\,\,\vert\beta_2\vert<\vert\alpha\vert}}\Vert\nabla_x^{\vert\beta_1\vert+1}\phi\Vert_{L^\infty_x}\left\{\Vert\nabla_v\partial^{\beta_2}_x\gamma\Vert_{L^2_{x,v}}+t\Vert\nabla_x\partial_x^{\beta_2}\gamma\Vert_{L^2_{x,v}}\right\}\Vert\partial_x^\alpha\gamma\Vert_{L^2_{x,v}}.
\end{split}
\end{equation}
Using \eqref{BootstrapAss}-\eqref{BootstrapAss1} and Gr\"onwall's Lemma, this remains bounded. The same estimates also work for $\abs{\alpha}=3$ as long as $\abs{\beta_1}<3$. If $\abs{\beta_1}=3$ (and thus $\alpha=\beta_1$, $\beta_2=0$) we estimate via \eqref{eq:DispLemmaL2} to obtain
\begin{equation}
  \abs{\iint \nabla_x\partial^{\alpha}_x\phi(x+tv)\cdot\left\{\nabla_v-t\nabla_x\right\}\gamma\cdot\partial_x^\alpha\gamma}\lesssim  \norm{\nabla_x\partial_x^\alpha\phi}_{L^2_{x}}(\norm{\nabla_v\gamma}_{L^\infty_x L^2_v}+t\norm{\nabla_x\gamma}_{L^\infty_xL^2_v})\norm{\partial_x^\alpha\gamma}_{L^2_{x,v}},
 \end{equation}
which is bounded by the bootstrap assumptions \eqref{BootstrapAss}-\eqref{BootstrapAss1}.

\subsection{Energy Estimates II: velocity regularity}\label{EEII}

Finally, the most difficult term comes from the $v$ derivative in \eqref{CommEq} and its higher order versions:
\begin{equation}\label{eq:vreg_base}
\begin{aligned}
\frac{1}{2}\frac{d}{dt}\Vert \partial_v^\alpha\gamma\Vert_{L^2}^2&=q\sum_{{\substack{\beta_1+\beta_2=\alpha\\\abs{\beta_2}<\abs{\alpha}}}}\iint_{\mathbb{R}^3_{x,v}}\nabla_x\partial_v^{\beta_1}\phi(x+tv)\cdot \partial_v^{\beta_2}\nabla_v\gamma(x,v)\cdot\partial^\alpha_v\gamma(x,v) dxdv\\
&\quad-q\sum_{{\substack{\beta_1+\beta_2=\alpha\\ \abs{\beta_2}<\vert\alpha\vert}}}t\iint_{\mathbb{R}^3_{x,v}}\nabla_x\partial_v^{\beta_1}\phi(x+tv)\cdot \partial_v^{\beta_2}\nabla_x\gamma(x,v)\cdot\partial^\alpha_v\gamma(x,v) dxdv.
\end{aligned}
\end{equation}
We treat the first term in \eqref{eq:vreg_base} using \eqref{BootstrapAss1}: If $\abs{\alpha}\leq 2$ we can directly estimate each summand
\begin{equation}
\begin{aligned}
 \left\vert\iint_{\mathbb{R}^3_{x,v}}\nabla_x\partial_v^{\beta_1}\phi(x+tv)\cdot \partial_v^{\beta_2}\nabla_v\gamma(x,v)\cdot\partial^\alpha_v\gamma(x,v) dxdv\right\vert&\lesssim t^{\abs{\beta_1}}\norm{\partial_x^{\beta_1}\nabla_x\phi}_{L^\infty_x}\norm{\gamma}_{L^2_xH^{\abs{\alpha}}_v}^2\\
 &\lesssim \varepsilon^2\ip{t}^{2\delta-2}\norm{\gamma}_{L^2_xH^{\abs{\alpha}}_v}^2,
\end{aligned} 
\end{equation}
and this also works when $\vert\alpha\vert=3$ and derivatives split: $\vert\beta_1\vert,\vert\beta_2\vert\le 2$. Finally, if all derivatives fall on $\nabla_x\phi$ we change variables
\begin{equation}
\begin{aligned}
 I&=\iint_{\mathbb{R}^3_{x,v}}\nabla_x\partial_v^{\alpha}\phi(x+tv)\cdot \nabla_v\gamma(x,v)\cdot\partial^\alpha_v\gamma(x,v) dxdv\\
 &=t^{-3}\iint_{\mathbb{R}^3_{x,a}}\nabla_x\partial_v^{\alpha}\phi(x)\cdot \nabla_v\gamma(a,\frac{x-a}{t})\cdot\partial^\alpha_v\gamma(a,\frac{x-a}{t}) dxda
\end{aligned} 
\end{equation}
and therefore, using \eqref{InterpolationZNorm} and \eqref{BootstrapAss1},
\begin{equation}
\begin{aligned}
 \vert I\vert&\lesssim t^{\vert\alpha\vert-3}\Vert\nabla_x\partial_x^\alpha\phi\Vert_{L^2_x}\Vert \partial_v^\alpha\gamma(a,\frac{x-a}{t})\Vert_{L^2_{x,a}}\Vert \nabla_v\gamma(a,\frac{x-a}{t})\Vert_{L^\infty_xL^2_a}\\
 &\lesssim t^{\vert\alpha\vert-\frac{3}{2}}\norm{\nabla_x\partial_x^{\alpha}\phi}_{L^2_x}\norm{\nabla_v\gamma}_{Z^\prime}\norm{\gamma}_{L^2_xH^{3}_v}\lesssim \varepsilon^2\ip{t}^{2\delta-2}\norm{\gamma}_{L^2_xH^{3}_v}^2.
\end{aligned} 
\end{equation}

The same considerations allow to control the second term in \eqref{eq:vreg_base} when $0\le t\le 1$. For $t\ge1$, due to the extra factor $t$ and the slow growth of $\norm{\gamma}_{L^2_xH^3_v}$, more care is needed. We let
\begin{equation}\label{MainIntegralTerm}
\begin{split}
I_{\beta_1,\beta_2}&:=\sum_{{\substack{\beta_1+\beta_2=\alpha\\\vert \beta_2\vert<\vert\alpha\vert}}}t\iint_{\mathbb{R}^3_{x,v}}\nabla_x\partial_v^{\beta_1}\phi(x+tv)\cdot \partial_v^{\beta_2}\nabla_x\gamma(x,v)\cdot\partial^\alpha_v\gamma(x,v) dxdv.
\end{split}
\end{equation}

\paragraph{\textbf{Case $\abs{\alpha}=1$.}} Here we necessarily have $\beta_1=\alpha$. Using that $\Vert\nabla_x\gamma\Vert_Z$ is uniformly bounded, we can then proceed as follows: We recognize that $I_{\alpha,0}$ can be written as (identifying an operator and its kernel)
\begin{equation}
\begin{split}
 I_{\alpha,0}&=t^2\iint_{\mathbb{R}^3}\gamma(y-tu,u)\gamma(y-tu,u)\mathcal{M}_{jk}(x-y)\partial_{x^j}\gamma(x-tv,v)\partial_{v^k}\gamma(x-tv,v)dxdydudv,\\
\mathcal{M}_{jk}&=(-\Delta)^{-1}\partial_j\partial_k=R_jR_k.
\end{split}
\end{equation}
Now changing variables, we can rewrite this as
\begin{equation}
\begin{split}
 I_{\alpha,0}&=t^{-4}\iint_{\mathbb{R}^3}\left\{\gamma(a,\frac{y-a}{t})\gamma(a,\frac{y-a}{t})\right\}\mathcal{M}_{jk}(x-y)\left\{\partial_{x^j}\gamma(b,\frac{x-b}{t})\partial_{v^k}\gamma(b,\frac{x-b}{t})\right\}dadb dxdy.
\end{split}
\end{equation}
Since $\mathcal{M}_{jk}$ is bounded as a map $L^2\to L^2$, using \eqref{GainDifference}, we see that
\begin{equation}
\begin{split}
\vert I_{\alpha,0}\vert&\lesssim t^{-4}\Vert \gamma(a,\frac{y-a}{t}) \gamma(a,\frac{y-a}{t})\Vert_{L^2_y L^1_a}\Vert \partial_{x^j}\gamma(b,\frac{x-b}{t})\partial_{v^k}\gamma(b,\frac{x-b}{t})\Vert_{L^2_xL^1_b}\\
&\lesssim t^{-1}\Vert \gamma\Vert_{Z'}\Vert\nabla_x\gamma\Vert_{Z'}\Vert \gamma\Vert_{L^2_{x,v}}\Vert\nabla_v\gamma\Vert_{L^2_{x,v}},
\end{split}
\end{equation}
and using \eqref{BootstrapAss} and integrating, we find the bound in \eqref{BootstrapCcl2}.

\medskip

\paragraph{\textbf{Case $\abs{\alpha}\ge2$.}} For higher $\vert \alpha\vert$, we use \eqref{DecR} to decompose
\begin{equation}
\begin{split}
 I_{\beta_1,\beta_2}&=t^{\vert\beta_1\vert+1}\int_{R=0}^\infty I^{\beta_1,\beta_2}_R\frac{dR}{R^2}\\
I^{\beta_1,\beta_2}_R&= \iint_{\mathbb{R}^3}\gamma(y-tu,u)\gamma(y-tu,u)\partial_x^{\beta_1}\partial_{x^j}\left\{\chi(R^{-1}\vert x-y\vert)\right\}\partial_{x^j}\partial_v^{\beta_2}\gamma(x-tv,v)\partial_{v}^\alpha\gamma(x-tv,v)dxdydudv.
\end{split}
\end{equation}
And we claim that, for $t\ge 100$,
\begin{equation}\label{ElementaryBounds}
\begin{split}
\vert I^{\beta_1,\beta_2}_R\vert&\lesssim R^{-1-\vert\beta_1\vert}\varepsilon^2\left[(\ln t)^{2\vert\alpha\vert-1}\varepsilon^2+(\ln t)^{-1}\Vert \gamma\Vert_{L^2_xH^{\vert\alpha\vert}_v}^2\right],\\
\vert I^{\beta_1,\beta_2}_R\vert&\lesssim R^2t^{-3-\vert\beta_1\vert}(t/R)^\frac{1}{2}\left[\varepsilon^2\Vert\gamma\Vert_{L^2_xH^{\vert\alpha\vert}_v}^2+\varepsilon\Vert \gamma\Vert_{L^2_{x}H^{\vert\alpha\vert}_v}^{3}\right],\qquad R\le t,\\
\vert I^{\beta_1,\beta_2}_R\vert&\lesssim_\delta Rt^{-2-\vert\beta_1\vert}\varepsilon^2\left[(\ln t)^{2\vert\alpha\vert-1}(\ln \ln t)^{4\vert\alpha\vert-1}\varepsilon^\frac{3}{2}+(\ln t\cdot\ln \ln t)^{-1}\Vert \gamma\Vert_{L^2_xH^{\vert\alpha\vert}_v}^2\right].
\end{split}
\end{equation}
We can combine these bounds and Gr\"onwall estimate to obtain the last energy bounds in \eqref{BootstrapCcl2}. We integrate the first bound for $R\ge t$, the second for $0\le R\le t/(\ln t)^{100}$ and the last for $t/(\ln t)^{100}\le R\le t$, to get
\begin{equation*}
\begin{split}
\frac{d}{dt}\Vert \gamma\Vert_{L^2_xH^{\vert\alpha\vert}_v}^2&\lesssim \varepsilon^4t^{-1}(\ln t)^{2\alpha}+\frac{\varepsilon^2}{t\ln t}\Vert \gamma\Vert_{L^2_xH^{\vert\alpha\vert}_v}^2+\frac{\varepsilon}{t(\ln t)^{50}}\Vert \gamma\Vert_{L^2_xH^{\vert\alpha\vert}_v}^{3}+\varepsilon^\frac{7}{2}t^{-1}(\ln t)^{2\vert\alpha\vert-1}(\ln \ln (t))^{4\vert\alpha\vert}
\end{split}
\end{equation*}
which lead to \eqref{BootstrapCcl2}.

\medskip

To get the first bound in \eqref{ElementaryBounds}, we use a crude estimate
\begin{equation*}
\begin{split}
\vert I^{\beta_1,\beta_2}_R\vert&\lesssim R^{-1-\vert\beta_1\vert}\Vert \gamma\Vert_{L^2_{x,v}}^2\Vert\partial_x\partial_v^{\beta_2}\gamma\Vert_{L^2_{x,v}}\Vert \partial_v^\alpha\gamma\Vert_{L^2_{x,v}}\lesssim R^{-1-\vert\beta_1\vert}\Vert \gamma\Vert_{L^2_{x,v}}^2\Vert \gamma\Vert_{L^2_{x}H^{\vert\alpha\vert}_{v}}^{1+\frac{\vert\beta_2\vert}{\vert\alpha\vert}}\Vert \gamma\Vert_{L^2_vH^{\frac{\vert\alpha\vert}{\vert\beta_1\vert}}_x}^\frac{\vert\beta_1\vert}{\vert\alpha\vert}\\
&\lesssim R^{-1-\vert\beta_1\vert}\varepsilon^{2+\frac{\vert\beta_1\vert}{\vert\alpha\vert}}\Vert \gamma\Vert_{L^2_xH^{\vert\alpha\vert}_v}^{1+\frac{\vert\beta_2\vert}{\vert\alpha\vert}}
\end{split}
\end{equation*}
and using convexity, this gives the first estimate in \eqref{ElementaryBounds}.

\medskip

On the other hand, we can change variables and integrate by parts to get
\begin{equation}
\begin{split}
I^{\beta_1,\beta_2}_R&= t^{-6}\iint_{\mathbb{R}^3}\gamma(a,\frac{y-a}{t})\gamma(a,\frac{y-a}{t})\partial_x^{\beta_1}\partial_{x^j}\left\{\chi(R^{-1}\vert x-y\vert)\right\}\partial_{x^j}\partial_v^{\beta_2}\gamma(b,\frac{x-b}{t})\partial_{v}^\alpha\gamma(b,\frac{x-b}{t})dxdydadb\\
&=t^{-6-\vert\beta_1\vert}\sum_{\theta_1+\theta_2=\beta_1,\,\,\theta_1\le\theta_2}c_{\theta_1,\theta_2}I^{\theta_1,\theta_2,\beta_2}_R,\\
I^{\theta_1,\theta_2,\beta_2}_R&:=\iint_{\mathbb{R}^3}\partial_v^{\theta_1}\gamma(a,\frac{y-a}{t})\partial_v^{\theta_2}\gamma(a,\frac{y-a}{t})\partial_{x^j}\left\{\chi(R^{-1}\vert z\vert)\right\}\partial_{x^j}\partial_v^{\beta_2}\gamma(b,\frac{z+y-b}{t})\partial_{v}^\alpha\gamma(b,\frac{z+y-b}{t})dzdydadb.
\end{split}
\end{equation}
In case $\theta_1=\beta_2=0$, $\theta_2=\beta_1=\alpha$, we see that
\begin{equation*}
\begin{split}
\vert I^{0,\alpha,0}_R\vert&\lesssim R^{2}t^3\Vert \gamma\Vert_{Z^\prime}\Vert \partial_x\gamma\Vert_{Z^\prime}\Vert \partial_v^\alpha\gamma\Vert_{L^2_{x,v}}^2\lesssim R^2t^3\varepsilon^2\Vert\partial_v^\alpha\gamma\Vert_{L^2_{x,v}}^2.
\end{split}
\end{equation*}
In case $\theta_1=0$, $\theta_2=\beta_1$, $\beta_2\ne 0$, H\"older's inequality gives
\begin{equation*}
\begin{split}
\vert I^{0,\theta_2,\beta_2}_R\vert&\lesssim R^{-1}\Vert \gamma(a,\frac{y-a}{t})\Vert_{L^\infty_yL^2_a}\Vert \partial_v^{\theta_2}\gamma(a,\frac{y-a}{t})\Vert_{L^6_yL^2_a}\Vert (\nabla\chi)(R^{-1}\vert x-y\vert)\Vert_{L^{\frac{6}{5}}_y}\\
&\qquad\times\Vert \partial_{x^j}\partial_v^{\beta_2}\gamma(b,\frac{x-b}{t})\Vert_{L^2_{x,b}}\Vert \partial_v^\alpha\gamma(b,\frac{x-b}{t})\Vert_{L^2_{x,b}}\\
&\lesssim R^{-1}\Vert \gamma\Vert_{Z^\prime}\cdot t^\frac{1}{2}\Vert \gamma\Vert_{L^2_xH^{\vert\beta_1\vert+1}_v}\cdot R^\frac{5}{2}\cdot t^3\cdot\Vert\nabla_x\partial_v^{\beta_2}\gamma\Vert_{L^2_{x,v}}\Vert \partial_v^\alpha\gamma\Vert_{L^2_{x,v}}\\
&\lesssim R^2t^3(t/R)^\frac{1}{2}\cdot \Vert \gamma\Vert_{Z^\prime}\cdot\Vert \gamma\Vert_{L^2_{x}H^{\vert\alpha\vert}_v}^{1+\frac{\vert\beta_1\vert+1}{\vert\alpha\vert}+\frac{\vert\beta_2\vert}{\vert\alpha\vert}}\Vert \gamma\Vert_{L^2_xH^{\vert\alpha\vert}_x}^{1-\frac{1}{\vert\alpha\vert}}\\
&\lesssim R^2t^3(t/R)^\frac{1}{2}\cdot\varepsilon^{2-\frac{1}{\vert\alpha\vert}}\Vert \gamma\Vert_{L^2_{x}H^{\vert\alpha\vert}_v}^{2+\frac{1}{\vert\alpha\vert}}
\end{split}
\end{equation*}
while if $\beta_2=0$, $\theta_1\ne0$, we proceed similarly
\begin{equation*}
\begin{split}
\vert I^{\theta_1,\theta_2,0}_R\vert&\lesssim R^{-1}\Vert \partial_v^{\theta_1}\gamma(a,\frac{y-z-a}{t})\Vert_{L^2_{y,a}}\Vert \partial_v^{\theta_2}\gamma(a,\frac{y-z-a}{t})\Vert_{L^6_zL^2_a}\Vert (\nabla\chi)(R^{-1}\vert z\vert)\Vert_{L^{\frac{6}{5}}_z}\\
&\qquad\times\Vert \nabla_x\gamma(b,\frac{y-b}{t})\Vert_{L^\infty_{y}L^2_{b}}\Vert \partial_v^\alpha\gamma(b,\frac{y-b}{t})\Vert_{L^2_{y,b}}\\
&\lesssim R^2t^3(t/R)^\frac{1}{2}\cdot\varepsilon^{2-\frac{1}{\vert\alpha\vert}}\Vert \gamma\Vert_{L^2_{x}H^{\vert\alpha\vert}_v}^{2+\frac{1}{\vert\alpha\vert}}.
\end{split}
\end{equation*}
Finally, if $\theta_1\ne0$, $\beta_2\ne 0$, then $\vert\alpha\vert=3$ and we obtain
\begin{equation}
\begin{split}
\vert I^{\theta_1,\theta_2,\beta_2}_R\vert&\lesssim R^2\Vert \partial_v\gamma(a,\frac{y-a}{t})\Vert_{L^6_yL^2_a}^2\Vert \partial_{x^j}\partial_v\gamma(b,\frac{z+y-b}{t})\Vert_{L^6_yL^2_b}\Vert \partial_{v}^\alpha\gamma(b,\frac{z+y-b}{t})\Vert_{L^2_{y,b}}\\
&\lesssim R^2t^3\Vert \gamma\Vert_{L^2_xH^2_v}^2\Vert \gamma\Vert_{H^1_xH^2_v}\Vert \gamma\Vert_{L^2_{x}H^3_v}\lesssim R^2t^3\varepsilon\Vert \gamma\Vert_{L^2_{x}H^3_v}^3.
\end{split}
\end{equation}

\medskip

We now obtain improved bounds in the regime $R\sim t$. In this case, we leave one additional derivative on the kernel to get
\begin{equation*}
\begin{split}
I^{\beta_1,\beta_2}_R&=t^{-5-\vert\beta_1\vert}\sum_{\theta_1+\theta_2+\theta_3=\beta_1,\,\,\theta_1\le\theta_2}c_{\theta_1,\theta_2,\theta_3}I^{\theta_1,\theta_2,\theta_3,\beta_2}_R,\\
I^{\theta_1,\theta_2,\theta_3,\beta_2}_R&=\iint_{\mathbb{R}^3}\partial_v^{\theta_1}\gamma(a,\frac{y-a}{t})\partial_v^{\theta_2}\gamma(a,\frac{y-a}{t})\partial_{x}^{\theta_3}\partial_{x_i}\left\{\chi(R^{-1}\vert x-y\vert)\right\}\partial_{x^j}\partial_v^{\beta_2}\gamma(b,\frac{x-b}{t})\partial_{v}^\alpha\gamma(b,\frac{x-b}{t})dxdydadb
\end{split}
\end{equation*}
with $\vert\theta_3\vert=1$, $\vert\theta_1\vert+\vert\theta_2\vert+\vert\theta_3\vert+\vert\beta_2\vert=\vert\alpha\vert$. If $\vert\theta_1\vert=\vert\theta_2\vert=0$, we can directly compute
\begin{equation*}
\begin{split}
\vert I^{0,0,\theta_3,\beta_2}_R\vert&\lesssim R\Vert\gamma\Vert_{Z^\prime}^2\Vert \gamma\Vert_{H^1_xH^{\vert\alpha\vert-1}_v}\Vert \gamma\Vert_{L^2_xH^{\vert\alpha\vert}_v}\lesssim R\varepsilon^2\Vert \gamma\Vert_{H^1_xH^{\vert\alpha\vert-1}_v}\Vert \gamma\Vert_{L^2_xH^{\vert\alpha\vert}_v}.
\end{split}
\end{equation*}
If $\vert\theta_1\vert=\vert\beta_2\vert=0$, then we can proceed similarly
\begin{equation*}
\begin{split}
\vert I^{0,\theta_2,\theta_3,0}_R\vert&\lesssim R\Vert\gamma\Vert_{Z^\prime}\Vert\nabla_x\gamma\Vert_{Z^\prime}\Vert \gamma\Vert_{L^2_xH^{\vert\alpha\vert-1}_v}\Vert \gamma\Vert_{L^2_xH^{\vert\alpha\vert}_v}\lesssim R\varepsilon^2\Vert \gamma\Vert_{H^1_xH^{\vert\alpha\vert-1}_v}\Vert \gamma\Vert_{L^2_xH^{\vert\alpha\vert}_v}.
\end{split}
\end{equation*}
Finally, if $\theta_2\ne 0$, $\beta_2\ne 0$, $\vert\alpha\vert=3$, $\vert\theta_1\vert=0$ and we decompose in Littlewood-Paley pieces:
\begin{equation*}
\begin{split}
I^{0,\theta_2,\theta_3,\beta_2}_R&=\sum_{C_1,C_2}I^{0,\theta_2,\theta_3,\beta_2}_{R,C_1,C_2},\\
I^{0,\theta_2,\theta_3,\beta_2}_{R,C_1,C_2}&:=\iint_{\mathbb{R}^3}\gamma(a,\frac{y-a}{t})\partial_v^{\theta_2}\gamma_{C_1}(a,\frac{y-a}{t})\partial_{x}^{\theta_3}\partial_{x_i}\left\{\chi(R^{-1}\vert x-y\vert)\right\}\partial_{x^j}\partial_v^{\beta_2}\gamma_{C_2}(b,\frac{x-b}{t})\partial_{v}^\alpha\gamma(b,\frac{x-b}{t})dxdydadb.\\
\end{split}
\end{equation*}
In case $\min\{C_1,C_2\}\le 1$, the derivative is favorable and one can proceed as follows. From now on, we may assume that the sums are over dyadic $C_1,C_2\ge 1$. Proceeding as above, we can bound
\begin{equation*}
\begin{split}
\vert I^{0,\theta_2,\theta_3,\beta_2}_{R,C_1,C_2}\vert&\lesssim Rt^3\Vert\gamma\Vert_{Z^\prime}\Vert\gamma\Vert_{L^2_xH^{\vert\alpha\vert}_v}\cdot\min\{C_2^{-1}\Vert\partial_v\gamma_{C_1}(a,\frac{y-a}{t})\Vert_{L^\infty_yL^2_a}\Vert\gamma\Vert_{H^1_xH^2_v}\,,\,C_1C_2\Vert \gamma_{C_1}\Vert_{L^2_{x,v}}\Vert\nabla_x\gamma\Vert_{Z^{\prime}}\}\\
\end{split}
\end{equation*}
using that
\begin{equation*}
\begin{split}
\Vert \partial_v\gamma_{C}(a,\frac{y-a}{t})\Vert_{L^\infty_yL^2_a}&\lesssim \min\{C\Vert \gamma\Vert_{Z^\prime},C^{-\frac{3}{2}}\Vert \gamma\Vert_{L^2_xH^3_v}\}\lesssim \min\{C\varepsilon,C^{-\frac{3}{2}}\Vert \gamma\Vert_{L^2_xH^3_v}\}\\
\Vert \gamma_C\Vert_{L^2_{x,v}}&\lesssim\min\{\Vert\gamma\Vert_{L^2_{x,v}},C^{-3}\Vert\gamma\Vert_{L^2_xH^3_v}\}\lesssim\min\{\varepsilon,C^{-3}\Vert\gamma\Vert_{L^2_xH^3_v}\}
\end{split}
\end{equation*}
and summing the bounds above and using interpolation, one finds that
\begin{equation*}
\begin{split}
\sum_{C_1\le C_2}\vert I^{0,\theta_2,\theta_3,\beta_2}_{R,C_1,C_2}\vert&\lesssim Rt^3\varepsilon\Vert\gamma\Vert_{L^2_xH^{\vert\alpha\vert}_v}\Vert \gamma\Vert_{H^1_xH^2_v}\sum_{C_1}\min\{\varepsilon,C_1^{-\frac{1}{2}}\Vert\gamma\Vert_{L^2_xH^3_v} \}\\
&\lesssim Rt^3\varepsilon^\frac{7}{3}\Vert\gamma\Vert_{L^2_xH^{\vert\alpha\vert}_v}^{\frac{5}{3}}\ln\langle\Vert\gamma\Vert_{L^2_xH^{\vert\alpha\vert}_v}\rangle,\\
\sum_{C_2< C_1}\vert I^{0,\theta_2,\theta_3,\beta_2}_{R,C_1,C_2}\vert&\lesssim Rt^3\varepsilon^2\Vert\gamma\Vert_{L^2_xH^{\vert\alpha\vert}_v}\sum_{C_1}\min\{\varepsilon C_1^2,C_1^{-1}\Vert\gamma\Vert_{L^2_xH^3_v}\}
\end{split}
\end{equation*}
In total, this gives, for any $\delta>0$,
\begin{equation*}
\begin{split}
\vert I^{0,\theta_2,\theta_3,\beta_2}_R\vert&\lesssim Rt^3\varepsilon^{2}\left\{\frac{1}{\ln \langle t\rangle\cdot\ln\langle\ln\langle t\rangle\rangle}\Vert\gamma\Vert_{L^2_xH^{\vert\alpha\vert}_v}^2+\varepsilon^\frac{3}{2}(\ln\langle t\rangle)^5(\ln\langle \ln \langle t\rangle\rangle)^{11}\right\}
\end{split}
\end{equation*}
which leads to an acceptable contribution in \eqref{BootstrapCcl2}

\qed

\section{Nonlinear analysis II: asymptotic flow and strong convergence}\label{NAII}
Once we have isolated the scattering mass in \eqref{ScatteringMass}, we can simplify the dynamics along rays by studying the electric field $\nabla_x\phi$. We compute
\begin{equation}
\begin{aligned}
\nabla_x\phi(x+tv)&=-\frac{1}{4\pi}\iint_{\mathbb{R}^3}\frac{x+tv-y}{\vert x+tv-y\vert^3}\gamma^2(y-tu,u)dudy=\frac{1}{4\pi}\frac{1}{t^3}\iint_{\mathbb{R}^3}\frac{z}{\vert z\vert^3}\gamma^2(a,\frac{x-a+z}{t}+v)dadz.
\end{aligned}
\end{equation}
The main contribution to this will come from
\begin{equation}\label{eq:defE}
\begin{split}
E_{main}(v,t)&:=\frac{1}{4\pi}\frac{1}{t^3}\iint_{\mathbb{R}^3}\frac{z}{\vert z\vert^3}\gamma^2(a,\frac{z}{t}+v)dadz=\frac{1}{4\pi}\frac{1}{t^2}\int_{\mathbb{R}^3}\frac{\zeta}{\vert \zeta\vert^3}m_t(\zeta-v)d\zeta.
\end{split}
\end{equation}
This expression only involves the scattering mass which converges. We thus define
\begin{equation*}
\widetilde{E}(v):=\frac{1}{4\pi}\int_{\mathbb{R}^3}\frac{\zeta}{\vert \zeta\vert^3}m_\infty(\zeta-v)d\zeta.
\end{equation*}
Note that $m_\infty\in L^1\cap L^\infty$, so that $\widetilde{E}(v)$ is well defined and $E_{main}=t^{-2}\widetilde{E}+o(t^{-2})$. Inspired by the characteristics of
\begin{equation}\label{VPRes}
\begin{split}
\partial_tf(x,v,t)&=qE(v,t)\cdot\left\{\nabla_v-t\nabla_x\right\}f(x,v,t)
\end{split}
\end{equation}
we define for $t\ge1$,
\begin{equation}\label{eq:def_sig}
 \sigma(x,v,t)=\gamma(X,v,t),\qquad X:=x+\ln(t)\cdot \widetilde{E}(v).
\end{equation}

\begin{proposition}
 Let $\gamma(x,v,t)$ be a global solution of \eqref{VPNew} as in Proposition \ref{prop:sol}. Then with $X,\sigma$ as in \eqref{eq:def_sig} there exists $\sigma_\infty(x,v)\in Z\cap H^1_{x,v}$ such that
 \begin{equation}
  \lim_{t\to\infty}\sigma(x,v,t)=\sigma_\infty(x,v) \quad\text{ in }Z\cap H^1_{x,v}.
 \end{equation}
\end{proposition}

\begin{proof}
In the following, we may assume $t\ge100$. From \eqref{eq:def_sig} we compute that
\begin{equation}\label{NewDer}
\begin{split}
\partial_{x^j}\sigma&=\partial_{x^j}\gamma,\qquad
\partial_{v^j}\sigma=\partial_{v^j}\gamma+\ln(t)\cdot \partial_{v^j}\widetilde{E}^k \partial_{x^k}\gamma,\qquad
\partial_t\sigma=\partial_t\gamma+\frac{q}{t}\widetilde{E}^p\partial_{x^p}\gamma,
\end{split}
\end{equation}
from which we obtain the equation
\begin{equation}\label{eq:sigma}
 \partial_t\sigma(x,v)=q\partial_{x^k}\phi(X+tv)(\partial_{v^k}\gamma)(X,v)+q\left\{\frac{1}{t}\widetilde{E}^p-t\partial_{x^p}\phi(X+vt)\right\}\cdot\partial_{x^p}\gamma(X,v).
\end{equation}
We claim that this is integrable in time in both $Z$ and $H^1_{x,v}$. 

\medskip

We start with the first term in \eqref{eq:sigma}. For $1\le j,k\le 3$, we compute
\begin{equation}
\begin{aligned}
 \norm{\partial_{x^k}\phi(X+tv)(\partial_{v^k}\gamma)(X,v)}_{Z}&\lesssim \norm{\nabla\phi}_{L^\infty}\cdot\norm{\nabla_v\gamma}_{Z},\\
 \norm{\partial_{x^j}[\partial_{x^k}\phi(X+tv)(\partial_{v^k}\gamma)(X,v)]}_{L^2_{x,v}}&\lesssim \norm{\partial_{x^j}\nabla\phi}_{L^\infty}\norm{\gamma}_{L^2_xH^1_v}+\norm{\nabla\phi}_{L^\infty}\norm{\gamma}_{H^1_xH^1_v},\\
 \norm{\partial_{v^j}[\partial_{x^k}\phi(X+tv)(\partial_{v^k}\gamma)(X,v)]}_{L^2_{x,v}}&\lesssim t\norm{\partial_{x^j}\nabla\phi}_{L^\infty}\norm{\gamma}_{L^2_xH^1_v}+\norm{\nabla\phi}_{L^\infty}\norm{\gamma}_{L^2_xH^2_v}\\
 &\qquad +\ln(t)\norm{\nabla\phi}_{L^\infty}\Vert\nabla_v\widetilde{E}(v)\Vert_{L^4}\norm{\nabla_v\gamma}_{H^1_xL^4_v}.
\end{aligned} 
\end{equation}
By boundedness of the Riesz transform, we see that
\begin{equation}
 \partial_{v^j}\widetilde{E}^k(v)=R_jR_k \tilde{m}_\infty\in L^p,\quad 1<p<\infty,\qquad \tilde{m}_\infty(x)=m_\infty(-x)
\end{equation}
so together with the bounds \eqref{BootstrapCcl2} on $\gamma$ and the decay \eqref{BootstrapAss1} of $\nabla\phi$, time integrability of the first term follows. For the second term in \eqref{eq:sigma} we compute that

\begin{equation}\label{eq:keysigma}
\begin{aligned}
\mathcal{DE}&:=\frac{1}{t}\widetilde{E}^p-t\partial_{x^p}\phi(X+vt)=\frac{1}{t}\nabla(-\Delta)^{-1}\left\{M_1+M_2+M_3\right\},\\
M_1(\zeta)&:=m_\infty(\zeta)-m_t(\zeta),\qquad
M_2(\zeta):=\int_{\mathbb{R}^3}\left\{\gamma^2(a,\zeta)-\gamma^2(a,\zeta-\frac{a}{t})\right\}da\\
M_3(\zeta)&:=\int_{\mathbb{R}^3}\left\{\gamma^2(a,\zeta-\frac{a}{t})-\gamma^2(a,\zeta+\frac{X-a}{t})\right\}da.
\end{aligned}
\end{equation}
We will often use the convolution structure. Sobolev inequality directly gives that
\begin{equation*}
\begin{split}
\frac{1}{t}\Vert \nabla(-\Delta)^{-1}M_j\Vert_{L^\infty_x}&\lesssim t^{-1}\Vert M_j\Vert_{L^2\cap L^4}.
\end{split}
\end{equation*}
This allows to bound the contribution of $M_1$ using \eqref{ConvergenceScatteringMass}.
We can treat $M_2$ similarly since
\begin{equation*}
\begin{split}
\Vert\mathfrak{1}_{\{\vert a\vert\ge t^\frac{1}{2}\}}\left\{\gamma^2(a,\zeta)-\gamma^2(a,\zeta-\frac{a}{t})\right\}\Vert_{L^1_{\zeta,a}}&\lesssim t^{-\frac{1}{2}}\Vert\gamma\Vert_{L^2_{x,v}}\Vert x\gamma\Vert_{L^2_{x,v}}\lesssim t^{-\frac{1}{3}}\varepsilon^2,\\
\Vert \mathfrak{1}_{\{\vert a\vert\le t^\frac{1}{2}\}}\left\{\gamma^2(a,\zeta)-\gamma^2(a,\zeta-\frac{a}{t})\right\}\Vert_{L^1_{\zeta,a}}&\lesssim t^{-\frac{1}{2}}\Vert\gamma\Vert_{L^2_{x,v}}\Vert\nabla_v\gamma\Vert_{L^2_{x,v}}\lesssim t^{-\frac{1}{3}}\varepsilon^2,\\
\Vert \gamma^2(a,\zeta)-\gamma^2(a,\zeta-\frac{a}{t})\Vert_{L^\infty_{\zeta}L^1_a}&\lesssim\Vert \gamma\Vert_{H^2_{x,v}}^2\lesssim\varepsilon^2t^{2\delta}
\end{split}
\end{equation*}
For $M_3$, we observe the bounds
\begin{equation*}
\begin{split}
\Vert \gamma^2(a,\zeta-\frac{a}{t})-\gamma^2(a,\zeta+\frac{X-a}{t})\Vert_{L^1_{\zeta,a}}&\lesssim t^{-1}\vert X\vert\Vert\gamma\Vert_{L^2_{x,v}}\Vert \nabla_v\gamma\Vert_{L^2_{x,v}}\lesssim t^{\delta-1}\vert X\vert\varepsilon^2,\\
\Vert \gamma^2(a,\zeta-\frac{a}{t})-\gamma^2(a,\zeta+\frac{X-a}{t})\Vert_{L^1_{\zeta,a}}&\lesssim \Vert\gamma\Vert_{L^2_{x,v}}^2\lesssim\varepsilon^2,\\
\Vert \gamma^2(a,\zeta-\frac{a}{t})-\gamma^2(a,\zeta+\frac{X-a}{t})\Vert_{L^\infty_{\zeta}L^1_a}&\lesssim\Vert \gamma\Vert_{Z^\prime}^2\lesssim\varepsilon^2
\end{split}
\end{equation*}

\medskip

These bounds are enough to control the $Z$-norm. Indeed, we see that
\begin{equation*}
\begin{split}
\Vert \frac{1}{t}\nabla(-\Delta)^{-1}M_j\cdot\partial_{x^p}\gamma(X,v)\Vert_{Z}&\lesssim \Vert \frac{1}{t}\nabla(-\Delta)^{-1}M_j\Vert_{L^\infty}\Vert \partial_{x^p}\gamma(X,v)\Vert_Z
\end{split}
\end{equation*}
and we can use this when $j=1,2$, while for $M_3$, we use that
\begin{equation*}
\begin{split}
\Vert \frac{1}{t}\nabla(-\Delta)^{-1}M_3\cdot\partial_{x^p}\gamma(X,v)\Vert_Z&\lesssim \Vert \frac{1}{t}\nabla(-\Delta)^{-1}M_3\cdot\mathfrak{1}_{\{\vert X\vert\le t^\frac{1}{6}\}}\partial_{x^p}\gamma(X,v)\Vert_Z\\
&\qquad+\Vert \frac{1}{t}\nabla(-\Delta)^{-1}M_3\cdot\mathfrak{1}_{\{\vert X\vert\ge t^\frac{1}{6}\}}\partial_{x^p}\gamma(X,v)\Vert_Z\\
&\lesssim \varepsilon^2t^{-1-\frac{1}{20}}\Vert \partial_{x^p}\gamma(X,v)\Vert_Z+\varepsilon^2t^{-1-\frac{1}{100}}\Vert \vert x\vert^\frac{1}{8}\nabla_x\gamma\Vert_Z
\end{split}
\end{equation*}
and again, this gives an acceptable contribution using \eqref{LocalizationInterpolationBound} below. The control of $L^2_vH^1_x$ is similar since
\begin{equation*}
\begin{split}
&\partial_{x^k}\left\{\left\{\frac{1}{t}\widetilde{E}^p-t\partial_{x^p}\phi(X+vt)\right\}\cdot\partial_{x^p}\gamma(X,v)\right\}\\
=&\frac{1}{t}\sum_j\partial_{x^p}(-\Delta)^{-1}M_j\cdot\partial_{x^p}\partial_{x^k}\gamma(X,v)-t \partial^2_{x^px^k}\phi(X+tv)\partial_{x^p}\gamma(X,v)
\end{split}
\end{equation*}
 with a new term that can be treated as follows:
\begin{equation*}
\begin{split}
t\Vert \nabla^2\phi(X+tv)\nabla_x\gamma(X,v)\Vert_{L^2_{x,v}}&\lesssim t\Vert\nabla^2\phi\Vert_{L^\infty}\Vert \nabla_x\gamma\Vert_{L^2_{x,v}}\lesssim t^{\delta-2}\varepsilon^3.
\end{split}
\end{equation*}

Finally, the control of $L^2_xH^1_v$ follows along similar lines, but requires a little more care. Indeed
\begin{equation*}
\begin{split}
&\partial_{v^k}\left\{\left\{\frac{1}{t}\widetilde{E}^p-t\partial_{x^p}\phi(X+vt)\right\}\cdot\partial_{x^p}\gamma(X,v)\right\}\\
=&\frac{1}{t}\sum_j\partial_{x^p}(-\Delta)^{-1}M_j\cdot\partial_{x^p}\partial_{v^k}\gamma(X,v)+\frac{\ln t}{t}\nabla(-\Delta)^{-1}M_j\cdot\partial_{x^p}\partial_{x^\ell}\gamma(X,v)\cdot\partial_{v^k}\widetilde{E}^\ell\\
&+\frac{1}{t}\sum_j\partial_{x^p}\partial_{x^k}(-\Delta)^{-1}M_j\cdot\partial_{x^p}\partial_{v^k}\gamma(X,v)
\end{split}
\end{equation*}
The last term is slightly singular. We can use the boundedness of the Riesz transform to control
\begin{equation*}
\begin{split}
\Vert \partial_{x^p}\partial_{x^k}(-\Delta)^{-1}M_j\cdot\partial_{x^p}\partial_{v^k}\gamma(X,v)\Vert_{L^2_{x,v}}&\lesssim \Vert M_j\Vert_{L^4_{v}}\Vert \partial_{x^p}\partial_{v^k}\gamma(X,v)\Vert_{L^4_vL^2_x}
\end{split}
\end{equation*}
and this is enough for $M_1,M_2$, and for $M_3$, we use the same decomposition to get
\begin{equation*}
\begin{split}
\Vert \partial_{x^p}\partial_{x^k}(-\Delta)^{-1}M_3\cdot\partial_{x^p}\partial_{v^k}\gamma(X,v)\Vert_{L^2_{x,v}\{\vert X\vert\le t^{\frac{1}{10}}\}}&\lesssim \varepsilon^2t^{-\frac{1}{100}}\Vert \partial_{x^p}\partial_{v^k}\gamma(X,v)\Vert_{L^4_vL^2_x}\\
\Vert \partial_{x^p}\partial_{x^k}(-\Delta)^{-1}M_3\cdot\partial_{x^p}\partial_{v^k}\gamma(X,v)\Vert_{L^2_{x,v}\{\vert X\vert\ge t^{\frac{1}{10}}\}}&\lesssim \Vert M_3\Vert_{L^6_v}\Vert \partial_{x^p}\partial_{v^k}\gamma(X,v)\Vert_{L^3_vL^2_x\{\vert X\vert\ge t^{\frac{1}{10}}\}}
\end{split}
\end{equation*}
and we can bound the last term with \eqref{LocalizationInterpolationBound}.

To finish the proof, it suffices to show that
\begin{equation}\label{LocalizationInterpolationBound}
\begin{split}
 \Vert \vert x\vert^\frac{1}{8}\nabla_{x}\gamma\Vert_{Z}+ \Vert \vert x\vert^\frac{1}{8}\nabla_{x,v}\gamma\Vert_{L^2_{x,v}}\lesssim \Vert \vert x\vert^\frac{1}{8}\gamma\Vert_{H^1_xH^{\frac{13}{8}}_v}&\lesssim \Vert x\gamma\Vert_{L^2}+\Vert \gamma\Vert_{H^3_{x,v}}\lesssim\varepsilon t^\delta.
\end{split}
\end{equation}
The first inequality follows from Sobolev embedding; the second inequality follows directly if $\gamma$ is supported on $\{\vert x\vert\le 1\}$ or is localized at small frequencies in $x$ or in $v$; in the other cases, we introduce a Littlewood-Paley decomposition as in \eqref{LPv} in $x$ ($P^x_A$) and in $v$ ($P^v_B$) to get
\begin{equation*}
\begin{split}
\Vert \vert x\vert^\frac{1}{8}\mathfrak{1}_{\{\vert x\vert\sim R\}}P^x_AP^v_B\gamma\Vert_{H^1_xH^{\frac{13}{8}}_v}&\lesssim R^\frac{1}{8} A B^\frac{13}{8}\Vert \mathfrak{1}_{\{\vert x\vert\sim R\}}P^x_AP^v_B\gamma\Vert_{L^2_{x,v}}\\
&\lesssim R^\frac{1}{8} A B^\frac{13}{8}\min\{R^{-1},A^{-3},B^{-3}\}\cdot\left[\Vert x\gamma\Vert_{L^2_{x,v}}+\Vert \gamma\Vert_{H^3_{x,v}}\right]
\end{split}
\end{equation*}
and we can sum this over dyadic $A,B,R\gtrsim 1$.

\end{proof}

\end{document}